\documentclass[psamsfonts]{amsart}

\usepackage{amssymb,amsfonts}
\usepackage{soul}
\usepackage[all,arc]{xy}
\usepackage{enumerate}
\usepackage{mathrsfs}
\usepackage{color}
\usepackage{tikz}
\usepackage{dsfont}
\usetikzlibrary{matrix}
\newtheorem{thm}{Theorem}[section]
\newtheorem{cor}[thm]{Corollary}
\newtheorem{prop}[thm]{Proposition}
\newtheorem{lem}[thm]{Lemma}

\theoremstyle{definition}

\theoremstyle{remark}
\newtheorem{rem}[thm]{Remark}

\newcommand{\QQ}{\mathbb{Q}}
\newcommand{\HH}{\mathrm{H}}

\newcommand{\CC}{\mathbb{C}} 
\newcommand{\PP}{\mathbb{P}}

\newcommand{\uconf}{\mathrm{UConf}} 
\newcommand{\pconf}{\mathrm{PConf}} 
\newcommand{\OO}{\mathscr{O}}
\newcommand{\LL}{\mathscr{L}}
\title{Stable cohomology of  the universal degree $d$ hypersurface in $\PP^n$}

\author{Ishan Banerjee}

\begin{document}

\begin{abstract}
	Let $U_{d,n}^*$ be the universal degree $d$ hypersurface in $\PP^n$. In this paper we compute the stable (with respect to $d$) cohomology of $U_{d,n}^*$ and give a geometric description of the stable classes. This builds on work of Tommasi \cite{T} and Das \cite{D}.
\end{abstract}
\maketitle
\section{Introduction}
Let $U_{d,n}$ be the \emph{parameter space} of smooth degree $d$ hypersurfaces in $\PP^n$. There is a natural inclusion $U_{d,n} \subseteq \PP^{\binom{n+d}{d}} = \PP (V_{d,n})$, where $V_{d,n}$ is the vector space of homogenous degree $d$ complex polynomials in $n+1$ variables. Let $$U_{d,n}^* := \{(f,p) \in U_{d,n} \times \PP^n | f(p) = 0 \}.$$ Let  $\phi: U_{d,n}^* \to U_{d,n}$  be defined by $\phi (f,p) =f$. The map $\phi: U_{d,n}^* \to U_{d,n}$ is the \emph{universal family} of smooth degree $d$ hypersurfaces in $\PP^n$; it satisfies the following property:
given a family $\pi: E \to B$  of smooth degree $d$ hypersurfaces in $\PP^n$ there is a unique diagram:

$$
\xymatrix{
E\ar[r]^{\exists !} \ar[d] & U_{d,n}^* \ar[d]\\
B \ar[r]^{\exists !} & U_{d,n}\\
}
$$
In other words, any family of smooth degree $d$ hypersurfaces is pulled back from this one. Our main result is as follows:
\begin{thm}\label{m1}
Let $d,n\ge 1$.  Then there is an embedding  of graded algebras:  \vspace{1 pt} $$  \phi:\mathrm{H}^*(\mathrm{PGL}_{n+1}(\CC); \QQ) \otimes \QQ [x]/(x^n) \hookrightarrow H^*(U_{d,n}^*; \QQ)$$
\vspace{1 pt} where $|x|=2$. 
\begin{enumerate}
    \item The element $\phi(x)  = c_1(\LL)$ where $\LL$ is the fiberwise canonical  bundle(defined in Section 2).
    \item Suppose $d\ge 4n+1$. Then $\phi$ is surjective in degree less than $\frac{d-1}{2}$.
\end{enumerate}

\end{thm}

We now define some spaces related to $U_{d,n}$, where we will prove similar results.  Let $X_{d,n} \subseteq V_{d,n}$ be the open subspace of polynomials defining a nonsingular hypersurface. The complement of $X_{d,n}$ in $V_{d,n}$ is known as the \emph{discriminant hypersurface}; it is the zero locus of the classical discriminant polynomial. It is known to be highly singular.

A point of $X_{d,n}$ determines a projective hypersurface up to a scalar. There is a natural action of $\CC^*$ on $X_{d,n}$. Let $U_{d,n} = X_{d,n}/ \CC^*$ be the quotient of this action. We can further quotient to obtain $M_{d,n} := U_{d,n}/ PGL_{n+1}\CC $, the \emph{moduli space} of degree $d$ smooth hypersurfaces in $\PP^n$. Let $$X_{d,n}^* := \{ (f,p) | f \in X_{d,n}, p\in \PP^n , f(p) =0\}.$$

There is a forgetful map $\pi: X_{d,n}^* \to X_{d,n}$  defined by $\pi(f,p) =f$. The fibres of $\pi$ are $$Z(f) := \pi^{-1}(f) = \{p \in \PP^n| f(p) =0\}\subseteq \PP^n.$$ It is well known that the  map $\pi$ is a fibre bundle.

$X_{d,n}^*$ also has several interesting quotients. The action of $GL_{n+1}$ on $X_{d,n}$ lifts to one on $X_{d,n}^*$. We obtain $U_{d,n}^* = X_{d,n}^*/\CC^* $. The map $\pi: X_{d,n}^* \to X_{d,n}$ is $\CC^*$-equivariant and descends to the map $\phi: U_{d,n}^* \to U_{d,n}$. We also have $M_{d,n}^* = X_{d,n}^* / GL_{n+1}(\CC)$.

We can rewrite our result in terms of $X_{d,n}^*$ and $M_{d,n}^*$ as well. This is important to us as our proof will mostly involve understanding the space $X_{d,n}^*$. The space  $M_{d,n}^*$ is important conceptually.

\begin{thm}\label{main}
Let $d,n\ge 1$.
\begin{enumerate}
    \item There is an embedding  of graded algebras:   $$  \psi:(\mathrm{H}^*(\mathrm{GL}_{n+1}(\CC); \QQ) \otimes \QQ [x]/(x^n)) \hookrightarrow \mathrm{H}^*(X_{d,n}^*; \QQ)$$ where $|x|=2$. 

    \item There is an embedding  of graded algebras:   $$  \varphi: \QQ [x]/(x^n) \hookrightarrow \mathrm{H}^*(M_{d,n}^*; \QQ)$$ where $|x|=2$. 
\end{enumerate}
Suppose that $d \ge 4n+1$. Then, the maps $\psi$ and $\varphi$ are surjective  in degree $\le \frac{d-1}{2}$.

\end{thm}
Theorem \ref{main} is equivalent to Theorem \ref{m1} after applying Theorem 2 of \cite{PS}.

\smallskip \noindent\textbf{Nature of stable cohomology:}
Throughout the course of the proof of Theorem \ref{main} we also obtain the following description of the stable cohomology classes of $X_{d,n}^*$-  the stable classes are tautological in the following sense:
There is a line bundle $\LL$ on $M_{d,n}^*$ defined by taking the canonical bundle fibrewise (we rigorously define $\LL$ in Section 2). We will show that $c_1(\LL), \dots, c_1(\LL)^{n-1}$ are nonzero in $H^*(M_{d,n}^*;\QQ)$ and that stably the entire cohomology ring of $M_{d,n}^*$ is just the algebra generated by $c_1(\LL)$. By \cite{PS}, $$\HH^*(X_{d,n}^* ; \QQ) \cong \HH^*(GL_{n+1}(\CC); \QQ) \otimes \HH^*(M_{d,n}^*; \QQ).$$ In this way we have some qualitative understanding of the stable cohomology of $X_{d,n}^*$.

Both the statement of Theorem \ref{main} and our proof of it are heavily influenced by \cite{T}, in which Tommasi proves analogous theorems for $X_{d,n}$. Our techniques and approach are  also similar to that of Das in \cite{D},where he proves $$\HH^*(X_{3,3}^*; \QQ) \cong \HH^*(GL_3(\CC);\QQ) \otimes \QQ[x]/x^3$$ with $|x| =2$.

In some sense, this paper shows that in a stable range, something similar to Das's theorem is true for marked hypersurfaces in general.

\smallskip \noindent\textbf{Method of Proof.}
One could attempt to prove Theorem \ref{main} by applying the Serre spectral sequence to the fibration $\pi: X_{d,n}^* \to X_{d,n}$. To successfully do this however, one would need to understand the groups $\HH^p(X_{d,n}; \HH^q(Z(f); \QQ))$. While we do a priori understand what the groups $\HH^p(X_{d,n};\QQ)$ are (This is the main theorem of \cite{T}), this is not sufficient for us to understand what the groups $\HH^p(X_{d,n}; \HH^q(Z(f); \QQ))$ are,  since $\HH^q(Z(f);\QQ)$ is a \emph{nontrivial} local coefficient system.
Instead we use an idea of Das and compute $\HH^*(X_{d,n}^p;\QQ)$, where $X_{d,n}^p := \{f \in X_{d,n}| f(p) =0\}$ to avoid any computations with nontrivial coefficient systems. After we have proved Theorem \ref{main} we can use it to deduce what these twisted cohomology groups are.
\begin{cor} \label{twisted}
Let $d,n >0$. Suppose$d \ge 4n+1$ and $k < \frac{d-1}{2}$.Then
\begin{equation*}
    \HH^k(X_d; \HH^{n-1}(Z(f); \QQ)) =
    \begin{cases}
      \HH^k(X_d; \QQ) & \textrm{if } n \textrm{ is odd}\\
      0        & \textrm{if } n \textrm{ is even}
    \end{cases}
  \end{equation*}
\end{cor}

\smallskip \noindent\textbf{Acknowledgements}: I'd like to thank my advisor Benson Farb for his endless patience and encouragement. I'd like to thank Eduard Looijenga for  help  with Lemma \ref{v}. I'd like to thank Nir Gadish and Ronno Das for some comments on the paper. I'd like to thank Burt Totaro for catching an error in a previous version of the paper. Finally I'd like to thank Gal Porat for his help in editing this paper.

\section{A lower bound on $\HH^k(X_{d,n}^*)$}
We begin by noting that there is an embedding of algebras $\HH^k(GL_{n+1}(\CC)) \otimes \QQ [x]/(x^n) \hookrightarrow \HH^k(X_d^*)$ in the stable range. More precisely, we have the following:

\begin{prop}\label{atleast}
Let $ n \ge 0$, and let $d > n+1$. Then there is a natural embedding.
$$i:H^*(GL_{n+1}(\CC); \QQ) \otimes \QQ [x]/(x^n) \hookrightarrow H^*(X_{d,n}^*; \QQ)$$ where $|x| =2$. The inclusion is one of algebras.
\end{prop}
 \begin{proof}
 We  first define the \emph{fiberwise canonical bundle} $\LL$ over $X_{d,n}^*$ as follows:
$$ \LL = \{(f,p,v) | (f,p) \in X_d^*, v \in \wedge^{n-1}T^*_p(Z(f))\}.$$
Note that the bundle $\LL$ is actually pulled back from a bundle on $M_{d,n}^*: = X_{d,n}^*/ GL_{n+1}(\CC)$. 
By the same argument as in Theorem 1 of \cite{PS}, $$\HH^*(X_{d,n}^*;\QQ) \cong \HH^*(GL_{n+1}(\CC); \QQ) \otimes \HH^*(M_{d,n}^* (\CC); \QQ).$$ Let $i: \mathrm{GL}_{n+1}(\CC) \to X_{d,n}$ be an orbit map. More precisely, Theorem 1 of \cite{PS} states that  the natural map $$\pi ^*:\HH^*(M_{d,n}^*; \QQ) \to \HH^*(X_{d,n}^*;\QQ)$$ makes $\HH^*(X_{d,n}^*;\QQ)$  a free$\HH^*(M_{d,n}^*; \QQ)-$ module. with a basis given by some set $\{\alpha_i$\} such that the pullbacks $\{i^*(\alpha_i)\}$  give a basis of $H^*(GL_{n+1}(\CC); \QQ)$. But since $H^*(GL_{n+1}(\CC); \QQ)$ is a free graded commutative algebra, this forces $H^*(X_{d,n}^*;\QQ)$ to be isomorphic to $H^*(GL_{n+1}(\CC); \QQ) \otimes H^*(M_{d,n}^* (\CC); \QQ)$ as algebras.

If we restrict $\LL$ to a particular hypersurface $Z$, the bundle $\LL|_Z = \OO_Z(d- n -1)$. The chern class of $\LL|_Z$ satisfies the equality- $$c_1(\OO_Z(d- n - 1))= (d- n-1)c_1(\OO_Z(1))  = d(d- n -1) \omega_Z,$$ where $\omega_Z$ is the Kahler class of the variety $Z$.   This implies that for $d > n+1$, the classes $c_1(\LL)| Z, \dots, c_1^{n-1}(\LL)|Z$ are nonzero since $\omega_Z, \dots \omega_Z^{n-1}$ are nonzero. Now taking $x = c_1(\LL)$, this implies that $\HH^*(M; \QQ)$ contains a subalgebra isomorphic to $\QQ [x]/x^n$.
\end{proof}



 



\section{The space $X_{d}^p$ and the Vassiliev method}
Let us define the ordered and unordered configuration space of a space $X$ as we will need to consider thes in this section. Given a space $X$, the $n$th \emph{ordered configuration space} of $X$ denoted $\pconf_n X$ is $$\pconf_n X : =\{(x_1 \dots, x_n) \in X^n| \forall i \neq j, x+i \neq x_j\}.$$ There is a natural action of the symmetric group on $n$ letters $S_n$ on $X$ by permuting the coordinates. The quotient $\pconf_n X /S_n$ is called the nth \emph{unordered configuration space} and denoted $\uconf_n X$.
In order to understand $X_{d,n}$ we will first look at the cohomology of a related space. For a fixed point $p \in \PP^n$ we set $$X_{d}^p = \{f \in X_d | f(p) =0 \}.$$  Then  $$X_d^p \subseteq V_{d}^p = \{f \in V_d| f(p) =0 \}.$$ The space $V_{d}^p$  is a vector space. The complement of  $X_{d}^p$ in $V_{d}^p$ will be called $\Sigma_{d,p}$. We will compute its Borel-Moore homology and use Alexander duality to compute $H^*(X_d^p)$.

In what follows we will often refer to a group $G_p$ defined as follows: if $p \in \PP^n$, it is by definition a one-dimensional  subspace $p \subseteq \CC^{n+1}$. Choose a complementary subspace  $W \subseteq \CC^{n+1}$ (it is not unique, but we will fix a particular one). We let $G_p = GL(W)$.

Let $x_1, \dots x_n$ be  local coordinates in a neighbourhood $U$ containing $p$. Pick a local trivialisation $s$ of the line bundle $\OO(d)$ in $U$.  There is an induced map $$f^*:  T_0^*(\OO(d)_p)  \to T^*_p(\PP^n) .$$ Let us use our local coordinates to identify $T_0^*(\OO(d)_p) $ with $\CC$ and $T^*_p(\PP^n)$ with $\CC^n$.

Suppose $f \in X_d^p$. Then the map $f^*$ is nonzero because $f$ has a regular zero locus. Let  This defines a map 
$$\pi: X_d^p \to T^*_p(\PP^n) - \{0\} \cong\CC^n- 0$$ defined by $\pi(f) = f^*(1).$
\begin{prop}
The map $\pi: X_d^p \to \CC^n- 0$ is a fibration.
\end{prop}
\begin{proof}
The group $G_p$ acts  on $\PP^n$ fixing $p$. Therefore it acts on both $X_{d}^{p}$ and $\CC^{n} - \{0\}$. The the map $\pi$ is equivariant with respect to these actions. The map $\pi$ therefore is the pullback of a map from $\pi': X_{d}^{p} / G_p $ to $\CC^{n} - \{0\}/ G_p$. But $\CC^{n}- \{0\} / G_p$ is a point and since $\pi'$ is surjective it is a fibration. Since pullbacks of fibrations are fibrations, $\pi$ is a fibration.
\end{proof}
Let $X_v:=\pi^{-1}(v)$ and let $$V_v := \{f \in V_d | f^*(1) =v \}.$$
Clearly, $X_v \subseteq V_v$. Let $\Sigma_v:=V_v - X_v$. We will try to understand the Borel-Moore homology of $\Sigma_v$. 

To accomplish this the Vassiliev method\cite{V} will be applied.  The Vassiliev method to compute Borel-Moore homology  involves stratifying a space and using the  associated spectral sequence to compute its Borel Moore homology. The space $\Sigma_v$  will be stratified based on the points at which  a section $f$ is singular. The techniques used are very similar to that in  \cite{T} which contains many of the technical details.

We will now construct a cubical space $X$ which will be involved in understanding $\Sigma_v$. Let $N = \frac{d-1}{2}$. 
Let $I$ be a subset of $\{1, \dots, N-1\}.$ For $k< N$, let  $$X_I :=\{(f, p)| f \in \Sigma_v, p: I \to \PP^n, p(I) \subseteq \textrm{ Singular zeroes of  }f \}.$$
We define $$X_{I \cup \{N\}}:= \{(f, p)| f \in \Sigma_v, p: I \to \PP^n, p(I) \subseteq \textrm{ Singular zeroes of  }f , f\in \bar \Sigma^{ge N}\}. $$ If $I \subseteq J$ then we have a natural map from $X_J \to X_I$ defined by restricting $p$. This gives $X_{\cdot}$ the structure of a cubical space over the set $\{1, \dots, N\}$. We can take the geometric realization of $X_{\cdot}$ denoted by $|X|$. Then there is  a map $\rho: |X| \to \Sigma_v$, induced by the forgetful maps $X_I \to \Sigma_v$.

$|X|$ is topologised in a non-standard way so as to make $\rho$ proper. We topologise it as follows:
 in \cite{T}, a space $|\mathscr{X}|$ is constructed with a  map $\rho : |\mathscr{X}| \to \Sigma$. Here, $\Sigma = V_d - X_d$.   The topology on $|\mathscr{X}|$ is chosen carefully so as to make $\rho$ proper. The construction of $|\mathscr{X}|$ as a set identical to that of $|X|$ except we replace $\Sigma_v$ with $\Sigma$.   There is a natural inclusion $|X| \to |\mathscr{X}|$. We give $|X|$ the subspace topology along this map.

\begin{prop}
The map $\rho: |X| \to \Sigma_v$ is a proper homotopy equivalence.
\end{prop}
\begin{proof}
This proof is nearly identical to that of Lemma 15 in \cite{T}.
The properness of $\rho : Y \to \Sigma_v$ follows from the properness of $\rho: |\mathscr{X}| \to \Sigma$ and the properties of the subspace topology. In our setting having contractible fibres implies that the map $\rho$ is a homotopy equivalence, this follows by combining Theorem 1.1 and Theorem 1.2 of \cite{L}.
 We will now prove that the fibres are contractible. If $f \not \in \bar \Sigma_v^{\ge N}$, let $\{p_1, \dots, p_k\}$ be the singular zeroes of $f$. In this case the fibre $\rho^{-1}(f)$  is a simplex with vertices given by the images of the points $(f,x_i) \in X_{\{1\}} \times \Delta_{\{1\}}$.    Now suppose $f \in \bar \Sigma_v^{\ge N}$. In this case the fibre $\rho^{-1}(f)$ is a cone over the point $f \in X_{N} \times \Delta_{\{N\}}$.

\end{proof}

Now as in any geometric realization, $|X|$ is filtered by $$F_n = \mathrm{im} (\coprod_{|I| \le n} X_I\times \Delta_k).$$ The $F_n$ form an increasing filtration of $|X|$ , i.e. $F_1 \subseteq F_2 \dots F_n \subseteq F_{n+1}\subseteq \dots$ and $\cup_{n=1}^{\infty}F_n = |X|$.

\begin{prop}\label{is}
Let $d,n \ge 1$. Let $N = \frac{d-1}{2}$. For $k<N$, the space $F_k- F_{k-1}$ is a $\Delta_k^{\circ}$- bundle, over a vector bundle $B_k$ over $\uconf_k(\PP^n -p)$.
\end{prop}
\begin{proof}
The space $F_k - F_{k-1}$ consists of the interiors of  $k$ simplices labeled by $\{f, p_0 \dots p_k\}$.
Let $$B_k = \{(f ,\{p_0 \dots p_k\}) \in \Sigma_v \times \uconf_k (\PP^n - p)| p_i \textrm{ are singular zeroes of } f\}.$$

We have a map  $\phi:F_k - F_{k-1} \to B_k$, defined by $$\phi((f ,\{p_0 \dots p_k\}) , s_0, \dots, s_k) = (f ,\{p_0 \dots p_k\}).$$ The map $\phi$ expresses $F_k - F_{k-1}$ as a fibre bundle  over $B_k$ with $\Delta^\circ_k$ fibres, i.e we have a diagram as follows:

   $$ \xymatrix{
\Delta^\circ_k  \ar[r] & F_k - F_{k-1} \ar[d]\\
 & B_k\\
}$$

We have a map $B_k \to \uconf_k (\PP^n - p)$ defined by $\{f,p_0 \dots ,p_k \} \mapsto \{p_0, \dots p_k\}$. This is a vector bundle by Lemma 3.2 in \cite{VW}. 
\end{proof}

We have a one-dimensional local coefficient system denoted $\pm\QQ$ on $\uconf_k(\PP^n - p)$ defined in the following way:
Let $S_k$ be the symmetric group on $k$ letters. We have a homomorphism $\pi_1 \uconf_k(\PP^n - p) \to S_k$ associated to the covering space $\textrm{PConf}_k(\PP^n - p) \to \uconf_k(\PP^n - p)$. Compose this homomorphism with the sign representation $S_k \to \pm 1 = GL_1(\QQ)$ to obtain our local system.  
\begin{prop}\label{cohoeq}
Let $d,n \ge 1$. Let $e_d = \mathrm{dim}_{\CC}(V_v)$. For $k< \frac{d-1}{2}$, $$\bar H_*(F_k - F_{k-1}) \cong H_{*-(k + 2e_d -2(n+1)(k+1))}(\uconf_k(\PP^n -p), \pm \QQ).$$
\end{prop}
\begin{proof}
By Proposition \ref{is} the space $F_k - F_{k-1}$ is a bundle over $\uconf_k(\PP^n - p)$. This fact implies that $$\bar H_*(F_k - F_{k-1}) \cong H_{*-(k + 2e_d -2(n+1)(k+1))}(\uconf_k(\PP^n -p),  \QQ(\sigma)).$$ Here $\QQ(\sigma)$ is the local sytem obtained by the action of $\pi_1(\uconf_k(\PP^n -p))$ on the fibres $\bar H_k(\Delta_k^{\circ})$ where in this case $\Delta_k^{circ}$ is the open $k$ simplex corresponding to the fibres of the map $F_k - F_{k-1} \to B_k$. But one observes that the action of $\pi_1(\uconf_k(\PP^n -p))$ on this open simplex is by permutation of the vertices which imples that $\QQ(\sigma)= \pm \QQ$. 

\end{proof}

As with any filtered space, we have a spectral sequence with $E_1^{p,q} = \bar H_{p+q}(F_p - F_{p-1};\QQ)$ converging to $\bar H_*(Y;\QQ)$. Now for $p<N$ by Proposition \ref{cohoeq} $$E_1^{p,q} = \bar H_{q-(2e_d-2(n+1)(p+1))}(\uconf_p(\PP^n -p);\pm \QQ).$$ We would like to claim that $E_1^{N,q}$ doesn't matter in the stable range. To be more precise, we have the following:
\begin{lem}
Let $d,n \ge 1$. Let $N = \frac{d-1}{2}$. Let $k > 2e_d -N$.
Then, $\bar H_k(Y - F_N;\QQ) \cong \bar H_k(Y; \QQ) $ .
\end{lem}
\begin{proof}
 We first will try to bound the $\bar H_* (F_N;\QQ)$  and then use the long exact sequence of the pair. 
$F_N$ is the union of locally closed subspaces $$\phi_k = \{(f, x_1,\dots,x_k), p| f \in \Sigma^{\ge N}, x_i \textrm{ are singular zeroes of }f, p \in\Delta_k \}.$$ We have a surjection $\pi: \phi_k \to \uconf_k(\PP^n - p) $. This map $\pi$ is in fact a fibre bundle with fibres $\Delta^k \times \CC^{e_d-N(n+1)}$. The space $\uconf_k(\PP^n - p)$ is $kn$ dimensional. Therefore $$\bar H_*(\phi_k;\QQ) = 0  \textrm{ if }  * > 2(e_d -(n+1)N) +kn < 2e_d -N.$$ This implies that for all $k$, $\bar H_*(\phi_k;\QQ) = 0$, if $* > 2e_d -N$. This implies $\bar H_*(F_N;\QQ) = 0$, if $* > 2e_d -N$ .
By the long exact sequence in Borel Moore homology associated to the pair $F_N \hookrightarrow Y$, $\bar H_k(Y - F_N;\QQ) \cong \bar H_k(Y;\QQ) $ for $k > 2e_d -N$.
\end{proof}

\section{Interlude}

In \cite{T}, Tommasi proves the following result:
\begin{thm}[\cite{T}]\label{T}
Let $d,n\ge 1$. Let $f \in X_{d,n}$. Let $\psi: GL_{n+1}(\CC) \to X_{d,n}$ be the orbit map defined by $ \psi(g) = g\cdot f$. 

Then $\psi^*: H^k(X_{d,n},\QQ) \to H^k(GL_{n+1}(\CC),\QQ)$ is an isomorphism for $k < \frac{d+1}{2}$. 

\end{thm}

In this section we shall look at the proof of Theorem \ref{T} in \cite{T} and use it to prove an identity used later on in this paper.
One of the ingredients in the proof of Theorem \ref{T} is a Vassiliev spectral sequence.
We introduce a new convention, by letting  $h$ denote the dimension of $H$. We also define $Gr(p, n)$ to be the Grassmanian of $p$-planes in $\CC^{n}$. In what follows we shall need a few basic facts about $H_*(Gr(p,n);\QQ)$ and Schubert symbols. Let $$0=E_0 \subsetneq E_1 \dots \subsetneq E_{n-1} \subsetneq E_n = \CC^n $$ be a complete flag.
Given $U \in Gr(p,n)$, we can associate to it a sequence of numbers, $a_i =\textrm{dim} U \cap E_i$. These $a_i$ satisfy the following conditions: $$0 \le a_{i+1} -a_i \le 1,a_0 =0 \textrm{ and }  a_n =p.$$

Such sequences are called \emph{Schubert symbols}.
Let $\bold{a} = (a_0 \dots a_n)$. We call $\bold{a}$ a Schubert symbol if $0 \le a_{i+1} -a_i \le 1$,$a_0 =0$ and $ a_n =p$. 
Associated to each Schubert symbol $\bold{a}$ we have a subvariety $W_{\bold{a}}\subseteq{Gr(p,\CC^n)}$ defined as follows. We  define 

$$W_{\bold{a}} = \overline{\{U \subseteq \CC^n| \textrm{dim} U = a_i\}}.$$

The main result we will be using is the following.
\begin{thm}\label{schmain}
Let $\bold{a}$ be a Schubert symbol. The classes $[W_{\bold{a}}] \in H_*(Gr(p, n);\QQ)$ form a basis.
\end{thm}
For a proof of Theorem \ref{schmain} see page 1071 of \cite{KL}.

\begin{prop}\label{schsum}
Let $n$ be a positive integer. Then $$\sum_{k,p} h_k(Gr(p, \CC^n);\QQ) = 2^n.$$
\end{prop}
\begin{proof}
By Theorem \ref{schmain}, $$\sum_{k,p} h_k(Gr(p, \CC^n);\QQ) = \sum_p \#\{(a_0, \dots a_n)  |0 \le a_{i+1} -a_i \le 1, a_0=0, a_n = p\} $$

$$ = \#\{(a_0, \dots a_n)  |0 \le a_{i+1} -a_i \le 1, a_0=0\}
$$
$$
=\#\{(b_1, \dots b_n) \in \{0,1\} \}
.$$

The last equality follows because if we are given a sequence of $a_i$, we can uniquely obtain a sequence of $b_i$, by letting $b_i = a_i -a_{i-1}$. 
\end{proof}

Our main aim of this section is to prove the following technical result that is necessary for our purposes.
\begin{thm}\label{sum}
The Vassiliev spectral sequence in \cite{T} degenerates in the stable range: if $p <\frac{d+1}{2}$ and if $ q>0$, then $E_1^{p,q} \cong E^\infty_{p,q}$. 

Equivalently, for $k<\frac{d+1}{2}$, 
\begin{equation}\label{sumeq}
    \sum_{p} h_{2(p+1)(n+1) - p  -k -1 } (\uconf_p(\PP^n);\QQ) = h_k(GL_{n+1};\QQ)
\end{equation}
(These are the diagonal terms in the spectral sequence).
\end{thm}
\begin{proof}
We already know that $$\sum_{p } h_{2(p+1)(n+1) - p  -k -1 } (\uconf_p(\PP^n); \pm \QQ) \ge h_k(GL_{n+1};\QQ) $$
 because the left hand side of \eqref{sumeq} are the appropriate terms in a spectral sequence converging to the right hand side of \eqref{sumeq}.
 
 It suffices to prove that
 
 $$
\sum_k \sum_{p } h_{2(p+1)(n+1) - p  -k -1} (\uconf_p(\PP^n); \pm \QQ)
$$
$$
= \sum _k h_k(GL_{n+1};\QQ) = 2^{n+1}.
 $$
 
  Lemma 2 in \cite{V} states that: $$h_{2(p+1)(n+1) - p  -k -1 } (\uconf_p(\PP^n), \pm \QQ) = h_{2(p+1)(n+1) - p  -k -1  -p(p-1)}(Gr(p, \CC^{n+1}); \QQ).$$ 
 Therefore $$
\sum_k \sum_{p } h_{2(p+1)(n+1) - p  -k -1} (\uconf_p(\PP^n); \pm \QQ) = \sum_k\sum_p h_k(Gr_(p, \CC^{n+1}); \QQ).
$$
 By Proposition \ref{schsum} this is equal to $2^{n+1}$.
 
\end{proof}

\section{Computation}
We would  like to know what the groups $\bar H_*(\uconf_{k+1}(\PP^n - p); \pm \QQ)$ are. 
First note that  in \cite{V} Vassiliev proves that :
\begin{prop}[\cite{V}]\label{v}
Let $k,n>0$. Then,$$H_*(\uconf_{k}(\PP^n); \pm \QQ) \cong H_{*-(k)(k-1)}(Gr_k(\CC^{n+1}); \QQ).$$
\end{prop}

Also note that  in light of Theorem \ref{schmain} the homology of Grassmannians is  well understood in terms of Schubert cells.

Consider the long exact sequence in Borel Moore homology associated to $$\uconf_{k+1}(\PP^n - p) \subseteq \uconf_{k+1}(\PP^n) \hookleftarrow \uconf_k(\PP^n - p).$$ The last inclusion is defined by the map $\phi :\uconf_k(\PP^n - p) \to \uconf_{k+1}(\PP^n)$ where $\phi(\{x_1 \dots x_n\}) = \{x_1 \dots x_n,p\} $.

We obtain a long exact sequence as follows:
\begin{equation}\label{splitLES}
     \dots \to \bar H_*(\uconf_k(\PP^n - p);\pm \QQ) \to \bar H_*(\uconf_{k+1}(\PP^n);\pm \QQ) \to \bar H_*(\uconf_{k+1}(\PP^n - p);\pm \QQ) \to 
\dots
\end{equation}

\begin{prop}\label{split}
Let $k,n>0$. Then there is a canonical decomposition $$H_*(\uconf_{k+1}(\PP^n);\pm \QQ) \cong \bar H_*(\uconf_k(\PP^n - p);\pm \QQ) \oplus \bar H_*(\uconf_k(\PP^n - p);\pm \QQ), $$ due to the fact that \eqref{splitLES} splits.
\end{prop}
\begin{proof}

 Lemma 2 of  \cite{V} implies that \eqref{splitLES} decomposes into split short exact sequences, i.e. $$\bar H_*(\uconf_{k+1}(\PP^n);\pm \QQ) \cong \bar H_*(\uconf_k(\PP^n - p);\pm \QQ) \oplus \bar H_*(\uconf_k(\PP^n - p);\pm \QQ) .$$
\end{proof} 

\begin{rem}
In fact the $H_*(\uconf_k(\PP^n - p);\pm\QQ)$ has a basis  given by  Schubert symbols with $a_1 =0$.  
\end{rem}

 \begin{prop}\label{ifnodiff}
 If the Vassiliev spectral sequence has no  nonzero differentials and $k < \frac{d-1}{2}$, then $H^k(X_v) \cong H^k(G_p)$ as vector spaces. 
 \end{prop}
\begin{proof}
  Now in our spectral sequence we had  $E^{p,q}_{1}= \bar H_{q - (2e_d - 2(p+1)(n+1))}(\uconf_{p+1}(\PP^n - p); \pm \QQ)$.
 First collect all terms in the main diagonal, i.e. $$V := \oplus_{p+q =l}  \bar H_{q -( 2D_n - 2(p+1)(n+1))}(\uconf_{p+1}(\PP^n - p); \pm \QQ) $$

It will suffice to prove that  

\begin{equation}\label{mainsumeq}
     \textrm{dim} V = \sum_{p \le 2D_n - k} h_{2(p+1)(n+1) - p  -k -1} (\uconf_p(\PP^n - pt);\pm \QQ) = h^k(GL_n; \QQ ) .
\end{equation}

 Proposition \ref{sum} implies

\begin{equation}\label{ifnodiffeq}
    \sum_{p} h_{2(p+1)(n+1) - p  -k -1 } (\uconf_p(\PP^n); \pm \QQ) = h_k(GL_{n+1}; \QQ).
\end{equation} 

Proposition \ref{v} implies, $$h_{2(p+1)(n+1) - p  -k -1 } (\uconf_p(\PP^n); \pm \QQ) =0 \textrm{ if } p>n$$. 

So as long as $n < 2(D_n +n+1 ) - k$, 
 $$ \sum_{p \le 2(D_n +n+1 ) - k} h_{2(p+1)(n+1) - p  -k -1 } (\uconf_p(\PP^n); \pm \QQ) = \sum_{p } h_{2(p+1)(n+1) - p  -k -1 } (\uconf_p(\PP^n); \pm \QQ).$$
 But the condition $n < 2(D_n +n+1 ) - k$ is equivalent to $k <2(D_n +1) +n$, which is true if $k <N$. We have another equality from Proposition \ref{split}, $$ h_k(\uconf_p(\PP^n - pt); \pm \QQ) +h_{k}(\uconf_{p-1}(\PP^n -pt); \pm \QQ) = h_k (\uconf_p (\PP^n); \pm \QQ).$$

Plugging this into \eqref{ifnodiffeq} we have 
$$
h^k(GL_{n+1}; \QQ) =\sum h_{2(p+1)(n+1) - p  -k} (\uconf_p(\PP^n); \pm \QQ)
$$

$$=  \sum  h_{2(p+1)(n+1) - p  -k -1} (\uconf_p(\PP^n - pt); \pm\QQ) +  h_{2(p+1)(n+1) - p  -k -1} (\uconf_{p-1}(\PP^n- pt); \pm \QQ).
$$

We have the identity $$ h^k (GL_n; \QQ)  + h^{k -(2n+1)}(GL_n; \QQ)= h^k(GL_{n+1} ;\QQ).$$
This implies,
\begin{align}\label{eq}
    & \quad h^k (GL_n; \QQ)  + h^{k -(2n+1)}(GL_n; \QQ)   \\
& \quad = \sum_p  h_{2(p+1)(n+1) - p  -k -1} (\uconf_p(\PP^n - pt);  \QQ) +  h_{2(p+1)(n+1) - p  -k -1} (\uconf_{p-1}(\PP^n- pt); \QQ) .
\end{align}

Now we will try to prove \ref{mainsumeq} by induction on $k$.
For $k =0$, \eqref{mainsumeq} is trivial. By induction $$ h^{k -(2n+1)}(GL_n; \QQ) = \sum_p h_{2(p+1)(n+1) - p  -k -1} (\uconf_{p-1}(\PP^n- pt); \pm \QQ).$$
Putting this into \ref{eq} we obtain
$$\sum_{p } h_{2(p+1)(n+1) - p  -k -1} (\uconf_p(\PP^n - pt);\pm\QQ) = h^k(GL_n; \QQ). $$

\end{proof}

 Now we can look at the Serre Spectral sequence associated to the fibration $$X_v \hookrightarrow X_p \to \CC^n -0.$$
 We observe that if there are no nonzero differentials, then $$H^*(X_p; \QQ) \cong H^*(X_v; \QQ) \otimes \QQ[e_{2n-1}]/e_{2n-1}^2.$$ This is because the Serre spectral sequence degenerates  and since $\QQ[e_{2n-1}]/e_{2n-1}^2$ is a free graded commutative algebra the ring structure of the total space is forced to be the tensor product.

\begin{prop}
Let $d>0$ and $p\in \PP^n$. Then, $$H^*(X_{d,p}; \QQ) \cong H^*(G_p; \QQ) \otimes A$$, where $A$ is $H^*(X_d^p/G_p; \QQ)$.
\end{prop}
\begin{proof}
This follows immediately from Theorem 2 in \cite{PS}.
\end{proof}

 We will also need the following fact that is a special case of Lemma 2.6 in \cite{D}.
 
 \begin{prop}\label{tens}
Let $d>0$, $k < \frac{d-1}{2}$. Let $U_d^{*} = X_d^* / \CC^*$.
 Then $$H^*(X_d^*; \QQ)  \cong H^*(U_d^*; \QQ) \otimes \QQ[e_1]/(e_1^2),$$ where $|e_1| = 1$.
 \end{prop}

Proposition \ref{tens} implies if there are no  nonzero differentials in both our Vassiliev spectral sequence and in the Serre spectral sequence associated to the fibration $X_{d,n}^p \to \CC^n - 0$ then
$$H^*(U_{d,p}; \QQ) \cong H^*(G_p; \QQ) \otimes \QQ[e_{2n-1}] / (e_{2n-1}^2)$$ for $* < \frac{d-1}{2}$. In case there are nonzero differentials in either spectral sequence, then $H^*(U_{d,p}; \QQ) \cong H^*(G_p; \QQ)$ for $*< \frac{d-1}{2}$.

\section{Comparing fibre bundles}
In this section we finish the proof of Theorem \ref{main}.

\begin{proof}[Proof of Theorem \ref{main}]
We compare three related fibre bundles and their associated spectral sequences. This is similar to the Proof of Theorem 1.1 in \cite{D}.
\begin{equation}\label{fibre}
    \xymatrix{
PG_p := Stab_{PGL(n+1)} p  \ar[rd] \ar[r] & U_{d,p} \ar[rd]  \ar[r] & U_d \ar[rd] & \\
& PGL_{n+1}(\CC)  \ar[r] \ar[d] & U_{d}^*  \ar[r]\ar[d] & U_d\times \PP^n \ar[d]\\
& \PP^n \ar[r]^= & \PP^n \ar[r]^= &\PP^n\\
} \end{equation}

By Proposition \ref{ifnodiff} and Theorem 1 of \cite{PS} there are two possibilities for $H^*(U_{d,p}; \QQ)$: either $$H^*(U_{d,p}; \QQ) \cong H^*(PG_p; \QQ) \otimes \QQ[e_{2n-1}] / (e_{2n-1}^2) \cong \Lambda <u_1, u_3,\dots u_{2n-1}, e_{2n-1}>$$(exterior algebra) or $$H^*(U_{d,p}; \QQ) \cong H^*(PG_p; \QQ) = \Lambda <u_1, u_3, \dots u_{2n-1}>.$$

Suppose for the sake of contradiction that $H^*(U_{d,p}) = \Lambda <u_3, \dots u_{2n-1}>$ for $* < \frac{d-1}{2}$. In this case $H^*(U_{d,p}; \QQ) \cong  H^*(PG_p; \QQ)$ for $* < \frac{d-1}{2}$.
Then  since the homology of the base and the fibres are isomorphic, $H^*(U_d^*;\QQ) \cong H^*(PGL_{n+1}(\CC);\QQ)$ for $* < \frac{d-1}{2}$. However by Proposition \ref{atleast}, $$H^*(PGL_{n+1}(\CC); \QQ) \otimes \QQ[x]/x^n)\subseteq H^*(U_d^*;\QQ).$$  But $H^*(PGL_{n+1}(\CC); \QQ)$ does not contain a subalgebra isomorphic to $H^*(PGL_{n+1}(\CC);\QQ) \otimes \QQ[x]/x^n)$. This is a contradiction.

So we must be in the case where,  $$H^*(U_{d,p};\QQ) \cong H^*(PG_p;\QQ) \otimes \QQ[e_{2n-1}] / (e_{2n-1}^2).$$ 

Consider the Serre spectral sequence associated to the fibration $U_d^* \to \PP^n$. Its $E_2$ page has terms $$E_2^{p,q}=H^p(\PP^n, H^q(U_d^p; \QQ)) \cong H^p(\PP^n; \QQ) \otimes H^q(U_d^p; \QQ).$$
Now $$H^q(U_d^p; \QQ) \cong H^q(PG_p;\QQ) \otimes \QQ[e_{2n-1}]/(e_{2n-1}^2) .$$
Consider the trivial fibre bundle $U_d \times \PP^n \to \PP^n$. There is a natural inclusion of fibre bundles as shown in \eqref{fibre}. This induces a map of spectral sequences between the associated Serre spectral sequences. 

Note that any class $\alpha \in H^q(U_d^p; \QQ)$ that lies in the image of $H^q(U_d;\QQ)$ is mapped to zero under any differential thanks to the fact that all dfferentials are zero in the spectral sequence associated to a trivial fibration.
The only possible nonzero differential in the $E_2$ page of the Serre spectral sequence associated to the fibration $U_d^* \to \PP^n$ is $d(e_{2n-1})$.

Suppose for contradiction that $d (e_{2n-1}) =0$. This implies that $$H^k(U_d^*;\QQ) \cong (H^*(U_{d,p};\QQ) \otimes H^*(\PP^n;\QQ))_k = (H^*(PG_p ; \QQ) \otimes H^*(\PP^n, \QQ))_k $$ for $k <\frac{d-1}{2}$.

Let $p(t)$ be the Poincare polynomial of $U_d^*$. We already know that $H^*(U_d^*; \QQ) \cong H^*(PGL_{n+1}(\CC) ; \QQ) \otimes H^*(U_d^* /\mathrm{PGL}_ {n+1}(\CC); \QQ)$. So $ (1+t^3) \dots (1+t^{2n+1})| p(t)$. On the other hand, if  $d e_{2n-1} = 0$ then 
$$p(t) = (1+t^3) \dots (1+t^{2n-1})(1 +t^2 +t^4 \dots t^{2n} ) \mod t^{\frac{d-1}{2}}$$. If $d \ge 4n +1$, then this implies that $(1+t^{2n+1}) \not | p(t)$. This is a contradiction.

So we must have a differential killing the class in $H^{2n}(\PP^n, H^0(U_{d,p})); \QQ)$. The differential must come from from $e_{2n-1}$, i.e. $d (e_{2n-1}) = ax^n$ for some $a \in \QQ^*$. This (along with multiplicativity of differentials) determines all differentials and implies (1).
 By Proposition \ref{tens} (1) $\implies$ (2). By Theorem 1 of \cite{PS} $$H^*(X_{d,n}^*; \QQ) \cong H^*(M_{d,n}^*; \QQ )\otimes (H^*(GL_{n+1})(\CC);\QQ)$$. In light of this (2) $\implies$ (3).

\end{proof}
 Having finished the proof of Theorem \ref{main} we can prove Corollary \ref{twisted}.
 \begin{proof}[Proof of  Corollary \ref{twisted}]
 Consider the fibration:

   $$  \xymatrix{
Z(f) \ar[r] & X_d^* \ar[d]\\
& X_d\\
}
$$
  and its associated Serre spectral sequence whose $E_2$ page is of the form $$H^p(X_d; H^q(Z(f) ;\QQ)) \implies H^*(X_d^*;\QQ).$$ ByTheorem \ref{T} for $* < \frac{d+1}{2}$
  $$
  H^*(X_d; \QQ) \cong H^*(GL_{n+1}(\CC); \QQ).
  $$
  By Theorem \ref{main}, we know that the classes in the $E_2$ page corresponding to the group
  $H^p(GL_{n+1}(\CC); c_1(\LL)^q)$ survive till the $E^{\infty}$ page and in the stable range all other terms are killed by differentials. 
  
  Now suppose $n$ is even. Then the only other terms in the spectral sequence are of the form $H^p(X_d; H^{n-1}(Z(f);\QQ))$. However it is not possible for any such term to be in the image or in the preimage of a nonzero differential. This is because all other terms survive so any possible nonzero differential must be from $H^{p_1}(X_d; H^{n-1}(Z(f);\QQ))$ to  $H^{p_2}(X_d; H^{n-1}(Z(f);\QQ))$ for some choice of $p_1$ and $p_2$. However no differential is of bidegree $(p_2-p_1,0)$. This implies that $$H^p(X_d; H^{n-1}(Z(f);\QQ)) \cong 0.$$

  A similar argument shows that if $n$ is odd, $H^p(X_d; H^{n-1}(Z(f);\QQ)) \cong H^p(X_d; \QQ)$. Essentially the only difference between the even case and the odd case is that in the odd case we have a class $c_1(\LL)^{\frac{n-1}{2}} \in H^{n-1}(Z(f); \QQ)$. By Theorem \ref{main}, we know that $H^p(X_{d,n};\QQ c_1(\LL)^{\frac{n-1}{2}})$ survives till the $E^{\infty}$ page. An argument similar to that in the even case shows that $$H^p(X_d; H^{n-1}(Z(f);\QQ)) \cong H^p(X_d; \QQ c_1(\LL)^{\frac{n-1}{2}}).$$
  
 \end{proof}

\end{document}